\documentclass[pdflatex,sn-mathphys-num]{sn-jnl}

\usepackage{graphicx}%
\usepackage{multirow}%
\usepackage{amsmath,amssymb,amsfonts}%
\usepackage{amsthm}%
\usepackage{mathrsfs}%
\usepackage[title]{appendix}%
\usepackage{xcolor}%
\usepackage{textcomp}%
\usepackage{manyfoot}%
\usepackage{booktabs}%
\usepackage{algorithm}%
\usepackage{algorithmicx}%
\usepackage{algpseudocode}%
\usepackage{listings}%

\theoremstyle{thmstyleone}%
\newtheorem{theorem}{Theorem}

\newtheorem{Cor}[theorem]{Corollary}%

\theoremstyle{thmstyletwo}%
\newtheorem{remark}{Remark}%
\newtheorem{lem}{Lemma}%

\theoremstyle{thmstylethree}%
\newtheorem{definition}{Definition}%

\begin{document}

\title[Anisotropic elliptic equations]{A Class of Non-linear  Anisotropic Elliptic problems
with Unbounded Coefficients and Singular Quadratic Lower Order Terms}

 
\author*[1]{\fnm{Fessel} \sur{Ach-houd}}\email{fessel.achhoud@studenti.unime.it}

\author*[2,3]{\fnm{Hichem} \sur{Khelifi}}\email{h.khelifi@univ-alger.dz}

\affil[1]{\orgdiv{Dipartimento di Matematica e Informatica}, \orgname{University of Catania}, \orgaddress{\street{Viale A. Doria, 6}, \city{Catania}, \postcode{95125}, \state{Italy}}}

\affil[2]{\orgdiv{Department of Mathematics, Faculty of Sciences}, \orgname{University of Algiers 1, Benyoucef BENKHEDDA}, \orgaddress{\street{2 Rue Didouche Mourad}, \city{Algiers}, \postcode{16000}, \state{Algeria}}}

\affil[3]{\orgdiv{Laboratoire d'\'equations aux d\'eriv\'ees partielles non lin\'eaires et histoire des math\'ematiques}, \orgname{\'Ecole normale sup\'erieure}, \orgaddress{\street{B. P 92, Vieux Kouba}, \city{Algiers}, \postcode{16050}, \state{Algeria}}}


\abstract{In this work, we study the existence and regularity results of anisotropic elliptic equations with a singular lower order term that grows naturally with respect to the gradient and unbounded coefficients. We take up the following model problem
\begin{equation*}
\left\{\begin{array}{ll}-\displaystyle\sum\limits_{j\in J} D_{j}\left(\left[ 1+ u^{q}\right]\vert D_{j}u\vert^{p_{j}-2} D_{j}u\right)+\sum\limits_{j\in J}\frac{\vert D_{j}u\vert^{p_{j}}}{ u^{\theta}}=f& \hbox{in}\;\Omega, \\
u>0& \hbox{in}\;\Omega, \\
 u =0 & \hbox{on}\; \partial\Omega, \end{array} 
 \right.
\end{equation*}
$\Omega$ is a bounded domain in $\mathbb{R}^{N}$, $j\in J=\{1,2,\ldots,N\},$ $q>0$, $0< \theta<1$,  $2\leq p_{1}\leq p_{2}\leq... \leq p_{N}$ and $f\in L^{1}(\Omega)$. Our study's conclusions will depend on the values of $q$ and $\theta$.}




\maketitle

\maketitle

\section{introduction}

Anisotropic elliptic equations describe physical and biological phenomena where
diffusion or conductivity varies with direction, such as flow in heterogeneous
porous media \cite{R18}, epidemic propagation across irregular landscapes
\cite{R19}, and thermistor models exhibiting directional heat–electric
responses \cite{HHH1,HHH2,HHH3}. Singular elliptic equations also arise in many
applications, featuring blow-up behaviour, boundary singularities
\cite{CD,ZD}. These problems often require
techniques beyond the classical isotropic theory \cite{V}, especially when
combined with nonstandard growth conditions \cite{R1,R2}.  

Let \(\Omega \subset \mathbb{R}^{N}\) be a bounded open set, where \(N \geq 2\), and let 
\((p_{j}) = (p_{1}, \ldots, p_{N}) \in \mathbb{R}^{N}\) satisfy
\[
1 \leq p^{-} = \min_{j \in J}\{p_{j}\} \leq p^{+} = \max_{j \in J}\{p_{j}\} < \infty .
\]

The anisotropic Sobolev spaces \(W^{1,(p_{j})}(\Omega)\) and \(W_{0}^{1,(p_{j})}(\Omega)\) are defined respectively by
\[
\begin{aligned}
W^{1,(p_{j})}(\Omega)
&= \big\{ v \in W^{1,1}(\Omega) :\; D_{j}v \in L^{p_{j}}(\Omega)\ \text{for all } j \in J \big\},\\[2mm]
W_{0}^{1,(p_{j})}(\Omega)
&= \big\{ v \in W_{0}^{1,1}(\Omega) :\; D_{j}v \in L^{p_{j}}(\Omega)\ \text{for all } j \in J \big\}.
\end{aligned}
\]

Equivalently, \(W_{0}^{1,(p_{j})}(\Omega)\) is the closure of \(C_{0}^{\infty}(\Omega)\) with respect to the norm
\[
\|v\|_{W_{0}^{1,(p_{j})}(\Omega)}
= \sum_{j \in J} \|D_{j}v\|_{L^{p_{j}}(\Omega)}.
\]

Endowed with this norm, \(W_{0}^{1,(p_{j})}(\Omega)\) is a separable and reflexive Banach space. 
Its dual space is denoted by \(W^{-1,(p'_{j})}(\Omega)\), where for each \(j \in J\), \(p_{j}'\) is the conjugate exponent of \(p_{j}\), that is,
\[
\frac{1}{p_{j}} + \frac{1}{p_{j}'} = 1.
\]

This article's main goal is to investigate the following problem.
\begin{equation}
\label{(1)}
\left\{\begin{array}{ll}\displaystyle-\sum\limits_{j\in J} D_{j}\left(\left[a(x)+u^{q}\right]\vert D_{j}u\vert^{p_{j}-2} D_{j}u\right)+b(x)\sum\limits_{j\in J}\frac{\vert D_{j}u\vert^{p_{j}}}{ u^{\theta}}=f& \hbox{in}\;\Omega, \\
u>0& \hbox{in}\;\Omega, \\
 u  =0 & \hbox{on}\; \partial\Omega, \end{array}
 \right.
\end{equation}
where $\Omega$ is a bounded open subset of $\mathbb{R}^{N}$ ($N\geq 3$), $f$, $q$, $\theta$, and $p_{j}$ are satisfies  the following conditions 
\begin{align}
&\label{(2)}
q>0,\quad 0<\theta<1,\\
&\label{(3)}
f\geq0,\;\; f\not\equiv0,\quad f\in L^{1}(\Omega),\\
&\label{(4)}
2\leq p_{1}\leq p_{2}\leq ... \leq p_{N-1}\leq p_{N}\quad \mbox{and}\quad 2\leq \overline{p}:=N\left(\sum\limits_{j\in J}\frac{1}{p_{j}}\right)^{-1}<N.
\end{align}
For some positive numbers $\alpha,\beta,\mu$, and $\nu$, we assume in this study that $a,b$ are measurable functions satisfying the following conditions.
\begin{align}
&\label{(5)}
 0<\alpha\leq a(x)\leq \beta,\\
&\label{(6)}
0<\mu\leq b(x)\leq \nu.
\end{align}
The isotropic case of problem \eqref{(1)}, where $p_j$ is constant for all $j \in J$, has been extensively studied by several authors, with notable contributions in \cite{olivia, boca92, boca97, dalill}. Also, a problem like \eqref{(1)} without the singular term was studied in more general context in \cite{FR1,FA1}. 

A further motivation for this work comes from the isotropic analogue of problem \eqref{(1)}, studied in \cite{R0}. There, the authors considered
\begin{equation}
\label{P3}
\left\{\begin{array}{ll}\displaystyle-\mbox{div}\left(\left[a(x)+u^{q}\right]\nabla u\right)+b(x)\frac{\vert \nabla u\vert^{2}}{ u^{\theta}}=f& \hbox{in}\;\Omega, \\
u>0& \hbox{in}\;\Omega, \\
 u  =0 & \hbox{on}\; \partial\Omega, \end{array}
 \right.
\end{equation}
with $q>0$, $0<\theta<1$, and $f\in L^{1}(\Omega)$. They proved the existence of nonnegative solutions and derived precise summability results depending on the balance between the exponent of the  absorption term $q$ and the singular term's exponent $\theta$. 
The present work investigates the same type of phenomenon in an \textit{anisotropic setting}, where the diffusion operator involves directional exponents $(p_{j})$ and the harmonic mean $\overline{p}$ replaces the Laplacian structure. 

The problem \eqref{P3} in the case of $q = 0$ has been addressed by several authors, including \cite{A1, A2, A3, A4, G1, G2}.

 In the context of anisotropic elliptic problems like \eqref{(1)} with no absorption term (i.e., when $q=0$, $a\equiv 1$, and $b\equiv 0)$, Di Castro in \cite{R2} addressed the following problem:
\begin{equation*}
\left\{\begin{array}{ll}\displaystyle-\sum\limits_{j\in J} D_{j}\left(\vert D_{j}u\vert^{p_{j}-2} D_{j}u\right)=f& \hbox{in}\;\Omega, \\
 u  =0 & \hbox{on}\; \partial\Omega, \end{array}
 \right.
\end{equation*}
where $p_j$ increases with respect to $j$ such that $2 - \frac{1}{N} \leq \overline{p} < N$, and $f \in L^1(\Omega)$. The author demonstrated that this problem admits a  solution $u \in W_0^{1, (s_j)}(\Omega)$, where
\begin{align}
\label{HH}
1\leq s_{j}<\frac{N(\overline{p}-1)}{\overline{p}(N-1)}p_{j},\quad\forall\; j\in J.
\end{align}
and for all  $\Phi \in C_0^1(\Omega),$
\begin{equation}\label{HH1}
\sum\limits_{j\in J}\int_{\Omega}\vert D_{j}u\vert^{p_{j}-2}D_{j}u D_{j}\Phi\,dx =\int_{\Omega}f\Phi\,dx .
\end{equation}

Further, this problem has been studied under different assumptions on the data in \cite{Stamp, BO1, BO2, OR1, BO3}, particularly when $p_j$ is constant for all $j \in J$.

 In the case where $q = 0$ and with lower order terms with sign conditions, Di Castro in \cite{R1} explored the following  general anisotropic problem
\begin{equation}
\label{P2}
\left\{\begin{array}{ll}\displaystyle-\sum\limits_{j\in J} D_{j}\left(\vert D_{j}u\vert^{p_{j}-2} D_{j}u\right)+\sum\limits_{j\in J}g_{j}(x,u,\nabla u)=f& \hbox{in}\;\Omega, \\
 u  =0 & \hbox{on}\; \partial\Omega, \end{array}
 \right.
\end{equation}
where $f \in L^1(\Omega)$ and $g_j: \Omega \times \mathbb{R} \times \mathbb{R}^N \to \mathbb{R}$ is a Carath\'eodory function satisfying the following conditions
\begin{itemize}
\item[(H1)] $g_{j}(x,s,\xi)\mbox{sgn}(s)\geq0$ for a.e. $x\in \Omega$, for all $s\in \mathbb{R}$ and $\xi\in \mathbb{R}^{N}$;
\item[(H2)] $\vert g_{j}(x,s,\xi)\vert\leq h(\vert s\vert)\vert\xi_{j}\vert^{p_{j}}$, for a.e. $x\in \Omega$, for all $s\in \mathbb{R}$ and $\xi\in \mathbb{R}^{N}$, where $h:\mathbb{R}\rightarrow \mathbb{R}^{+}$ is a continuous, nondecreasing function such that $h(t)>\gamma>0$ for $\vert t\vert$ sufficiently large;
\item[(H3)] There exists $\nu>0$ such that $\vert g_{j}(x,s,\xi)\vert\geq \nu \vert \xi_{j}\vert^{p_{j}}$ for a.e. $x\in \Omega$, for all $s\in \mathbb{R}$ and $\xi\in \mathbb{R}^{N}$.
\end{itemize}
The author investigated the existence and regularity of solutions to this problem, yielding the following results
\begin{itemize}
\item[$\bullet$] If conditions (H1) and (H2) hold, and if $p_j$ satisfies the inequality
$$\frac{\overline{p}(N-1)}{N(\overline{p}-1)}<p_{j}<\frac{\overline{p}(N-1)}{N-\overline{p}},\quad \forall\; j\in J,$$
then there exists at least one  solution $u \in W_0^{1, (s_j)}(\Omega)$, where $s_j$ is defined in \eqref{HH} and \eqref{HH1} holds.
\item[$\bullet$] If conditions (H1), (H2), and (H3) hold, then a weak solution $u$ exists in $W_0^{1, (p_j)}(\Omega)$.
\end{itemize} 

On the other hand, when \(b = 0\) and \(q \neq 0\), we have recently addressed, see \cite{FH}, a corresponding problem involving a singular right-hand side of the type \(\displaystyle\frac{f}{u^{l}}\), with \(l > 0\).

We end this  brief history of the problem by recalling the paper \cite{bacii}, in their work, the authors establish the existence of weak solutions for a broad class of anisotropic Dirichlet problems of the type
\[
\mathcal{A}u + \Phi(x,u,\nabla u) = \Psi(u,\nabla u) + \mathcal{L}u + f \quad \text{in } \Omega,
\]
where \(\Omega \subset \mathbb{R}^N\) \((N \ge 2)\) is a bounded domain and \(f \in L^1(\Omega)\) is arbitrary.
The anisotropic principal operator is
\[
\mathcal{A}u = -\sum_{j=1}^N \partial_j\left( \vert\partial_j u\vert^{p_j - 2}\partial_j u \right),
\]
and the lower-order term \(\Phi\) is given by
\[
\Phi(x,u,\nabla u) = \left( 1 + \sum_{j=1}^N \mathfrak{a}_j \vert\partial_j u\vert^{p_j} \right)|u|^{m-2}u,
\]
with parameters \(m,p_j>1\), \(\mathfrak{a}_j \ge 0\), and \(\sum_{k=1}^N \frac{1}{p_k} > 1\).
A key novelty of the study is the inclusion of a possibly singular gradient-dependent nonlinearity
\[
\Psi(u,\nabla u) = \sum_{j=1}^N |u|^{\theta_j - 2}u , \vert\partial_j u\vert^{q_j},
\]
where \(\theta_j>0\) and \(0 \le q_j < p_j\).
The lower-order operator $\mathcal{L}$ belongs to the general class $\mathcal{BC}$ introduced in \cite{baci}. 
The class $\mathcal{BC}$ consists of \emph{bounded} operators
\[
\mathcal{L} : W_0^{1,(p_j)}(\Omega) \rightarrow W^{-1,(p_j')}(\Omega)
\]
that satisfy the following structural conditions:

\begin{itemize}
    \item[(P$_1$)]  
    The operator $\mathcal{A} - \mathcal{L}$ is coercive on $W_0^{1,(p_j)}(\Omega)$, namely,
    \[
    \frac{\langle \mathcal{A}u - \mathcal{L}u,\, u \rangle}{\|u\|_{W_0^{1,(p_j)}(\Omega)}} \to \infty 
    \qquad \text{as } \|u\|_{W_0^{1,(p_j)}(\Omega)} \to \infty .
    \]

    \item[(P$_2$)]  
    If $u_l \rightharpoonup u$ and $v_l \rightharpoonup v$ weakly in $W_0^{1,(p_j)}(\Omega)$, then
    \[
    \lim_{l \to \infty} \langle \mathcal{L}u_l,\, v_l \rangle
    = \langle \mathcal{L}u,\, v \rangle .
    \]
\end{itemize}

These assumptions allow the authors to include a broad family of lower-order terms while preserving the coercivity and compactness properties needed for their existence theory.

The authors obtain existence results under two different regimes:
\begin{enumerate}
\item when \(\theta_j > 1\) for every \(j = 1,\dots, N\);

\item when \(\theta_j \le 1\) for at least one index \(j\).
\end{enumerate}
   In the latter case, by assuming \(f \ge 0\) almost everywhere in \(\Omega\), they further prove the existence of nonnegative weak solutions.

%
 In our study, we face three main difficulties. The first comes from the anisotropic structure of the operator, which requires estimates adapted to different growth conditions in each direction. The second difficulty is the presence of the singular lower-order term depending on the gradient $\frac{\vert D_{j} u\vert^{p_{j}}}{\vert u\vert^{\theta}}$, which complicates the analysis, especially near the points where the unknown $u$ may approach zero. The third difficulty is related to the additional term 
 $$
\displaystyle-\sum\limits_{j\in J} D_{j}\left(u^{q}\vert D_{j}u\vert^{p_{j}-2} D_{j}u\right)$$ 
whose nonlinear form interacts with the previous terms and makes the compactness arguments more delicate.
\begin{definition}
\label{Definition3.1}
Let the condition in \eqref{(3)} hold. A function $u\in W_{0}^{1,1}(\Omega)$ is a distributional solution to problem \eqref{(1)} if
\begin{align}
\label{(13)}
& u>0\quad\quad \mbox{ a.e. in}\; \Omega,\\
\label{(14)}
& \sum\limits_{j\in J}\left[a(x)+u^{q}\right]\vert D_{j}u\vert^{p_{j}-2}D_{j}u\in (L^{\sigma_{j}}(\Omega))^N,\quad \mbox{for all}\; \sigma_{j}<p_{j}',\\
\label{(15)}
& \sum\limits_{j\in J} b(x)\frac{\vert D_{j}u\vert^{p_{j}}}{u^{\theta}}\in L^{1}(\Omega),
\end{align}
and for all  $\Phi \in W_0^{1,d_j}(\Omega),\; d_j>p_j$
\begin{align}
\label{(16)}
&\displaystyle\sum\limits_{j\in J}\int_{\Omega}\left[a(x)+u^{q}\right]\vert D_{j}u\vert^{p_{j}-2}D_{j}u D_{j}\Phi +\sum\limits_{j\in J}\int_{\Omega} b(x)\frac{\vert D_{j}u\vert^{p_{j}}}{u^{\theta}}\Phi =\int_{\Omega}f\Phi\,dx .
\end{align}
\end{definition}
\par Our main results are as follows
\begin{theorem}
\label{T2}
Assume \eqref{(2)}-\eqref{(6)} hold. Then, 
\begin{itemize}
\item[(i)] If $0<q<1-\theta$, there exists a distributional solution $u\in W_{0}^{1,(r_{j})}(\Omega)$ to problem \eqref{(1)}, such that 
\begin{align}
\label{(17)}
r_{j}=\frac{N(\overline{p}-\theta)}{\overline{p}(N-\theta)}p_{j},\quad \forall \; j\in J.
\end{align}
\item[(ii)] If $1-\theta<q\leq1$, there exists a distributional solution $u\in W_{0}^{1,(\eta_{j})}(\Omega)$ to problem \eqref{(1)}, such that 
\begin{align}
\label{(017)}
r_{j}<\frac{N(\overline{p}-1+q)}{\overline{p}(N+q-1)}p_{j},\quad \forall \; j\in J.
\end{align}
\item[(iii)] If $q>1$, there exists a distributional solution $u\in W_{0}^{1,(p_{j})}(\Omega)$ to problem \eqref{(1)}.
\end{itemize}
\end{theorem}
We now give a brief description of the proof of our main theorem. The first step consists in removing the degeneracy of the principal operator and the singularity of the lower-order gradient term. For this purpose, we consider a sequence of approximate problems associated with \eqref{(1)}, where both difficulties are regularized. In the second step, we establish suitable a priori estimates for the approximate solutions. The main challenge here is to control the term
$$-\sum_{j \in J} D_{j}\left(u^{q}\vert D_{j}u\vert^{p_{j}-2}D_{j}u\right)$$
in an appropriate anisotropic Lebesgue space. Finally, by using Fatou’s lemma together with a standard compactness argument, we pass to the limit in the approximate problems and obtain a weak solution to the original equation.
\begin{remark}
We collect here several observations related to the assumptions of Theorem~\ref{T2}.
  
We first note that hypothesis \eqref{(4)} already ensures that  
\[
1 < r_{j} \le p_{j} \qquad \text{for every } j \in J.
\]  
Moreover, by combining assumptions \eqref{(2)} and \eqref{(4)}, we obtain the stronger condition  
\[
r_{j} > p_{j} - 1, \qquad \forall\, j \in J.
\]
  
We also observe that the condition \(1 - \theta < q\) appearing in Theorem~\ref{T2}(ii) is equivalent to  
\[
\frac{N(\overline{p}-\theta)}{\overline{p}(N-\theta)} p_{j}  
\;<\;
\frac{N(\overline{p}-1+q)}{\overline{p}(N+q-1)} p_{j},
\qquad \forall\, j \in J.
\]  
This relation will be useful when comparing the integrability exponents arising in the proof.
  
Furthermore, the hypothesis \(q \le 1\) in Theorem~\ref{T2}(ii) automatically implies that  
\[
1 < \eta_{j} < p_{j},
\]  
which plays an important role in establishing the required estimates.
 
It is also worth mentioning that, in the isotropic case \(p_{j} = 2\) for all \(j \in J\), Theorem~\ref{T2} reduces to known regularity results for classical elliptic equations. In particular, our conclusions coincide with those of Theorem~1.1 in \cite{R0}.
  
Finally, by applying the Sobolev embedding theorem, we can describe the integrability of the solutions.  
For Theorem~\ref{T2}(i), the solution \(u\) also satisfies  
\[
u \in L^{\overline{r}^{*}}(\Omega),
\qquad 
\overline{r}^{*} = \frac{N(\overline{p}-\theta)}{N-\overline{p}}.
\]  
For Theorem~\ref{T2}(ii), the solution enjoys the additional property  
\[
u \in L^{s}(\Omega),
\qquad \text{for every }\; s < \frac{N(\overline{p}-1+q)}{\,N-\overline{p}\,}.
\]
\end{remark}
For ease of reading, Section 2 begins with a brief recall of the main analytical tools of anisotropic Sobolev spaces and the definition of the approximate problem. We then state the essential a priori estimates Lemmas as well as some convergences results, which form the main step on the proof, and finally we pass to the limit.
\section{Preliminaries}
Before proceeding on the proof  we'll refresh the reader's memory on some previously established lemmas  on the functional setting in which we look for solutions to our main problem \eqref{(1)}.
\begin{lem}\cite{R10}
\label{L1}
There exists a positive constant $\mathcal{C}_1$ that depends only on $\Omega$, such that for $v\in W_{0}^{1,(p_{j})}(\Omega)$, $\overline{p}<N$ we have
\begin{align}
\label{(7)}
&\displaystyle\left\Vert v\right\Vert_{L^{\delta}(\Omega)}\leq \mathcal{C}_1\prod_{j\in J}\Vert D_{j}v\Vert_{L^{p_{j}}(\Omega)}^{\frac{1}{N}}, \quad \forall \delta\in[1,\overline{p}^{*}], \\
\label{(8)}
&\displaystyle\left\Vert v\right\Vert_{L^{\overline{p}^{*}}(\Omega)}^{p_{N}}\leq \mathcal{C}_1 \sum_{j\in J}\Vert D_{j}v\Vert_{L^{p_{j}}(\Omega)}^{p_{j}},
\end{align}
where $\displaystyle\overline{p}^{*}=\frac{N\overline{p}}{N-\overline{p}}$, $\displaystyle\frac{1}{\overline{p}}=\frac{1}{N}\sum_{j\in J}\frac{1}{p_{j}}$.
\end{lem}
\begin{lem}\cite{R10}
\label{L2}
There exists a positive constant $\mathcal{C}_2$, depending only on $\Omega$, such that for $v\in W_{0}^{1,(p_{j})}(\Omega)\cap L^{\infty}(\Omega)$, $\overline{p}<N$. we have
\begin{equation}
\label{(9)}
\displaystyle\left(\int_{\Omega}\vert v\vert^{z}dx\right)^{\frac{N}{\overline{p}}-1}\leq \mathcal{C}_2\prod_{j\in J}\left(\int_{\Omega}
\vert D_{j}v\vert^{p_{j}}\vert v\vert^{t_{j}p_{j}}dx\right)^{\frac{1}{p_{j}}},
\end{equation}
for every $z$ and $t_{j}\geq 0$ selected in such a way that
\begin{equation*}
\begin{cases}
\displaystyle\frac{1}{z}=\frac{\gamma_{j}(N-1)-1+\frac{1}{p_{j}}}{t_{j}+1},\\
\displaystyle\sum_{j\in J}\gamma_{j}=1.
\end{cases}
\end{equation*}
\end{lem}
\par To prove that the solution $u$ in $\Omega$ is positive, we need the theorem, which we will discuss later. As a preliminary to this theorem, we take into account the following  problem
\begin{align}
\label{(10)}
\left\{\begin{array}{ll}-\displaystyle\sum\limits_{j\in J} D_{j}\left[\vert D_{j}u\vert^{p_{j}-2}D_{j}u\right]=\lambda \vert u\vert^{q-2}u& \hbox{in}\;\;\Omega, \\
 u  =0 & \hbox{on}\;\; \partial\Omega, \end{array}
 \right.
\end{align}
here $\lambda>0$ and $p_{1}<q<p_{N}$.
\par We define weak supersolutions for the given problem \eqref{(10)}; for further details, see \cite{R3}.
\begin{definition}
\label{D1}
A function $u\geq 0$, a.e. in $\Omega$, $u\in W_{0}^{1,(p_{j})}(\Omega)$, is a positive weak supersolution for \eqref{(10)} if it obeys the inequality
\begin{equation}
\label{(11)}
\displaystyle\sum_{j\in J}\int_{\Omega} \vert D_{j}u\vert^{p_{j}-2}D_{j}u D_{j}\phi \geq 0, \quad \forall \phi \in C_{0}^{\infty}(\Omega),\;\; \phi \geq 0.
\end{equation}
\end{definition}
The following weak Harnack inequality for weak super-solutions is the key result.
The estimate is of local type. For simplicity, we will work with cubes. In what follows, $K_{x_{0}}(\rho)$ denotes the cube in $\mathbb{R}$ with side length \(\rho\) and center \(x_0\), whose edges are parallel to the coordinate axes. When the center is understood, we will abbreviate the notation and write \(K(\rho)\).
\begin{theorem}{\cite{R3,R4}}
\label{T1}
Let $u$ be a weak non-negative super-solution of \eqref{(10)}, with $u< M$ in $\Omega$, and assume $p_{1}\geq2$. Then, there exists a positive constant $\mathcal{C}_3$ such that we have 
\begin{equation}
\label{(12)}
\displaystyle\rho^{-\frac{N}{\beta}}\Vert u\Vert_{L^{\beta}(K(2\rho))}\leq C \min\limits_{K(\rho)}u,
\end{equation}
 for $\displaystyle\beta<\frac{N(p_{1}-1)}{N-p_{1}}$, if $p_{1}\leq N$, for any $\beta$, if $p_{1}>N$.
\end{theorem}
\par The following is obvious from the previous Theorem.
\begin{Cor}{\cite{R3}}
\label{C1}
Assume that $p_{1}\geq2$ and let $u$ be a weak non-negative solution for \eqref{(10)}. Then $u$ is strictly positive in $\Omega$, or u is the trivial solution.
\end{Cor}
\par We shall employ the truncation $\mathcal{T}_{k}$, defined as follows if $k>0$
\begin{equation*}
\mathcal{T}_{k}(l)=\max (- k , \min ( l , k ))\quad \mbox{and}\quad \mathcal{G}_{k}(l)=l-\mathcal{T}_{k}(l).
\end{equation*}

\section{Proof of Theorem \ref{T2}}
\label{section3}
\subsection{Approximate problems}

In order to construct a distributional solution to problem \eqref{(1)}, we begin by
considering the following regularized Dirichlet problems: for every $n\in\mathbb{N}$,
\begin{equation}
\label{approx-problem}
\left\{
\begin{array}{ll}
\displaystyle
-\sum_{j\in J} D_{j}\left(\left[a(x)+(T_{n}(u_{n}))^{q}\right]
\vert D_{j}u_{n}\vert^{p_{j}-2}D_{j}u_{n}\right)&\\
\quad+
\displaystyle\sum_{j\in J} b(x)
\frac{u_{n}\vert D_{j}u_{n}\vert^{p_{j}}}{\left(\vert u_{n}\vert+\frac{1}{n}\right)^{\theta+1}}
= f_{n} & \text{in }\Omega,
\\
u_{n}=0 & \text{on }\partial\Omega,
\end{array}
\right.
\end{equation}
where $f_{n}=\mathcal{T}_{n}(f)$. Since $f\geq 0$ and $f\in L^{1}(\Omega)$, it follows that $f_{n}\in L^{\infty}(\Omega)$ and
$f_{n}\geq 0$.
The principal part of \eqref{approx-problem} is an anisotropic Leray--Lions
operator with measurable bounded coefficient. The remaining term
\begin{equation}
\label{lower-term}
B_{n}(u_{n})
=
\sum_{j\in J}
b(x)\frac{u_{n}\,\vert D_{j}u_{n}\vert^{p_{j}}}{\left(\vert u_{n}\vert+\frac{1}{n}\right)^{\theta+1}}
\end{equation}
is a lower--order perturbation having natural growth with respect to the
anisotropic gradient.
To apply Theorem~2.5 of A.\,Di Castro (\cite{R3}, PhD Thesis, Chapter~2),
we must verify that the term \eqref{lower-term} is pointwise dominated by a constant
multiple of $\vert D_{j}u_{n}\vert^{p_{j}}$.

Since $0<\mu\leq b(x)\leq \nu$, we estimate, for $l\geq0$,
\[
 b(x)\frac{l}{\left(l+\frac1n\right)^{\theta+1}}
=
b(x)\frac{n^{\theta+1}l}{\left(1+nl\right)^{\theta+1}}
\leq
\nu l.
\]
Therefore, we obtain, for each $j\in J$,
\begin{equation}
\label{growth-bound}
\left\vert
b(x)\frac{u_{n}\,\vert D_{j}u_{n}\vert^{p_{j}}}{\left(\vert u_{n}\vert+\frac{1}{n}\right)^{\theta+1}}
\right\vert
\leq
\nu u_n\,\vert D_{j}u_{n}\vert^{p_{j}}.
\end{equation}
This shows that the lower order term in \eqref{approx-problem}
has natural growth of order $p_{j}$ with respect to the anisotropic gradient,
precisely as required in Theorem~2.5 of A.\,Di Castro's thesis.
Moreover, the quantity in \eqref{lower-term} satisfies the sign condition
\[
b(x)\,\frac{u_{n}\,\vert D_{j}u_{n}\vert^{p_{j}}}{\left(\vert u_{n}\vert+\frac{1}{n}\right)^{\theta+1}}\,u_{n}
\geq 0.
\]
Which means that all the hypotheses of Theorem~2.5 in Di Castro's thesis are satisfied.
Hence, for each $n\in\mathbb{N}$, problem \eqref{approx-problem} admits a weak solution $u_{n}\in W^{1,p_{j}}_{0}(\Omega).$
%
%
Finally, by testing the equation with the negative part $u_{n}^{-}$,
using $f_{n}\geq0$ and the fact that lower order  term
is nonnegative, we obtain
\[
\sum_{j\in J}\int_{\Omega} \vert D_{j}u_{n}^{-}\vert^{p_{j}} = 0,
\]
and therefore $u_{n}^{-}=0$. Thus, $
u_{n}\geq0 \text{ in }\Omega$.
\par At end, equivalently, we stress that the solutions  $u_{n}$ of the problem \eqref{approx-problem} satisfies the following equations
\begin{equation}
\label{(21)}
\left\{\begin{array}{ll}\displaystyle-\sum\limits_{j\in J} D_{j}\left(\left[a(x)+u_{n}^{q}\right]\vert D_{j}u_{n}\vert^{p_{j}-2}D_{j}u_{n}\right)&\\
\quad\displaystyle+\sum\limits_{j\in J}b(x)\frac{u_{n}\vert D_{j}u_{n}\vert^{p_{j}}}{\left( u_{n}+\frac{1}{n}\right)^{\theta+1}}=f_{n}&\hbox{in}\;\Omega, \\
u_{n} =0 & \hbox{on}\; \partial\Omega, \end{array}
\right.
\end{equation}
in the sense that $u_{n}$ satisfies
\begin{align}
\label{(22)}
&\sum_{j\in J}\int_{\Omega}\left[a(x)+ u_{n}^{q}\right] \vert D_{j}u_{n}\vert^{p_{j}-2}D_{j}u_{n}D_{j}\varphi\,dx +\sum_{j\in J}\int_{\Omega}b(x)\frac{u_{n}\vert D_{j}u_{n}\vert^{p_{j}}}{\left(  u_{n}+\frac{1}{n}\right)^{\theta+1}}\varphi dx\nonumber\\
&\quad=\int_{\Omega}f_n\varphi dx, \quad \forall \varphi\in W_{0}^{1,(p_{j})}(\Omega)\cap L^{\infty}(\Omega).
\end{align}
\subsection{Uniform estimates }
$C$ will denote a constant (not dependent on $n$) that may vary from line to line in this subsection.
We proceed, now, by proving suitable a priori estimates for the singular and degenerate terms.
\begin{lem}
\label{lem1}
Assume that the assumptions  \eqref{(2)}-\eqref{(6)} hold true. Then, the sequence
$u_n$ satisfies the following estimates
\begin{align}
\label{(23)}
&\displaystyle\alpha\sum\limits_{j\in J}\int_{\Omega}\frac{u_{n}\vert D_{j} u_{n}\vert^{p_{j}}}{\left(u_{n}+\frac{1}{n}\right)^{\theta+1}}\leq \int_{\Omega} f ,\\
\label{(24)}
&\displaystyle\frac{1}{k}\sum_{j\in J}\int_{\Omega}\left[a(x)+u_{n}^{q}\right]\vert D_{j}\mathcal{T}_{k}(u_{n})\vert^{p_{j}}\leq \int_{\Omega}f\frac{\mathcal{T}_{k}(u_{n})}{k},
\end{align}
for every $n\in \mathbb{N}-\{0\}$, and for every $k>0$.
\end{lem}
\begin{proof}
Let us consider $\frac{\mathcal{T}_{k}(u_{n})}{k}$, $k>0$, as a test function in \eqref{(21)}. Using \eqref{(5)} and \eqref{(6)}, we obtain
\begin{align}
\label{(25)}
&\frac{1}{k}\sum_{j\in J}\int_{\Omega}\left[a(x)+u_{n}^{q}\right]\vert D_{j}\mathcal{T}_{k}(u_{n})\vert^{p_{j}}+ \alpha\sum_{j\in J}\int_{\Omega}\frac{u_{n}\vert D_{j} u_{n}\vert^{p_{j}}}{\left(u_{n}+\frac{1}{n}\right)^{\theta+1}}\frac{\mathcal{T}_{k}(u_{n})}{k}\nonumber\\
&\quad \leq\int_{\Omega} \mathcal{T}_{n}(f) \frac{\mathcal{T}_{k}(u_{n})}{k}.
\end{align}
Since $\mathcal{T}_{n}(f)\leq f$, $\frac{\mathcal{T}_{k}(u_{n})}{k}\leq1$, and by removing the operator term on the left, we have
\begin{align}
\label{(26)}
\alpha\sum_{j\in J}\int_{\Omega}\frac{u_{n}\vert D_{j} u_{n}\vert^{p_{j}}}{\left(u_{n}+\frac{1}{n}\right)^{\theta+1}}\frac{\mathcal{T}_{k}(u_{n})}{k}\leq\int_{\Omega} f.
\end{align}
By allowing $k$ to tend to $0$ in \eqref{(26)}, we can deduce \eqref{(23)} using Fatou's Lemma.
However, by dropping the nonnegative second term of \eqref{(25)}, we obtain
\begin{align}
\label{(27)}
\frac{1}{h}\sum_{j\in J}\int_{\Omega}\left[a(x)+u_{n}^{q}\right]\vert D_{j}\mathcal{T}_{k}(u_{n})\vert^{p_{j}} &\leq\int_{\Omega} \mathcal{T}_{n}(f) \frac{\mathcal{T}_{k}(u_{n})}{k}\nonumber\\
&\leq \int_{\Omega} f \frac{\mathcal{T}_{k}(u_{n})}{k},
\end{align}
hence, \eqref{(24)} holds true.
\end{proof}
Next, as a consequence of the previous lemma, we state the following lemma.
\begin{lem}
\label{L5}
Assume that the assumptions  \eqref{(2)}-\eqref{(6)} hold true. Then
\begin{itemize}
\item[(i)] If  $0<q\leq 1-\theta$, then, the sequence $u_{n}$ is uniformly bounded in $W_{0}^{1,(r_{j})}(\Omega)$, where $r_{j}$ define in \eqref{(17)} for all $j\in J$.
\item[(ii)] If $1-\theta<q\leq1$, then, the sequence $u_{n}$ is uniformly bounded in $W_{0}^{1,(r_{j})}(\Omega)$, with $\eta_{j}$ satisfies \eqref{(017)} for all $j\in J$.
\item[(iii)] If $q>1$, then, the sequence $u_{n}$ is uniformly bounded in $W_{0}^{1,(p_{j})}(\Omega)$ for all $j\in J$.
\end{itemize}
Moreover, the sequence $\{\mathcal{T}_{k}(u_{n})\}_{n}$ is uniformly bounded in  $W_{0}^{1,(p_{j})}(\Omega)$ for all $k>0$ and all $j\in J$. 
\end{lem}
\begin{proof}
\begin{itemize}
\item[i)]  For \(0 < q \le 1 - \theta\), we first note that on the set \(\{u_n \ge 1\}\), the inequality $$\frac{1}{u_{n}^{\theta}}\leq 2^{\theta+1}\frac{u_{n}}{\left(u_{n}+\frac{1}{n}\right)^{\theta+1}}$$ holds. Combining this with estimates \eqref{(23)}--\eqref{(6)} yields
\begin{align*}
\sum_{j\in J}\int_{\{u_{n}\geq 1\}}\frac{\vert D_{j}u_{n}\vert^{p_{j}}}{u_{n}^{\theta}}&\leq 2^{\theta+1}\sum_{j\in J}\int_{\{u_{n}\geq 1\}}\frac{u_{n}\vert D_{j} u_{n}\vert^{p_{j}}}{\left(u_{n}+\frac{1}{n}\right)^{\theta+1}}\leq \frac{2^{\theta+1}}{\mu}\Vert f\Vert_{L^{1}(\Omega)},
\end{align*}
thus
\begin{align}
\label{(28)}
\int_{\{u_{n}\geq 1\}}\frac{\vert D_{j}u_{n}\vert^{p_{j}}}{u_{n}^{\theta}}&\leq  C\quad \forall\; j\in J.
\end{align}
By using H\"older inequality with exponents $\frac{p_{j}}{r_{j}}$ and $\left(\frac{p_{j}}{r_{j}}\right)'$, \eqref{(28)} and the fact that $u_{n}\leq \mathcal{G}_{1}(u_{n})+1$ on the set $\{u_{n}\geq1\}$, we obtain for all $j\in J$
\begin{align}
\label{(29)}
\int_{\Omega}\vert D_{j}\mathcal{G}_{1}(u_{n})\vert^{r_{j}}&=\int_{\{u_{n}\geq 1\}}\frac{\vert D_{j}\mathcal{G}_{1}(u_{n})\vert^{\eta_{j}}}{u_{n}\frac{\theta r_{j}}{p_{j}}}u_{n}^{\frac{\theta r_{j}}{p_{j}}}\nonumber\\
&\leq \left(\int_{\{u_{n}\geq1\}}\frac{\vert D_{j}u_{n}\vert^{p_{j}}}{u_{n}^{\theta}}\right)^{\frac{r_{j}}{p_{j}}}\left(\int_{\{u_{n}\geq 1\}}u_{n}^{\frac{\theta \eta_{j}}{p_{j}-r_{j}}}\right)^{\frac{p_{j}-r_{j}}{p_{j}}}\nonumber\\
&\leq C^{\frac{r_{j}}{p_{j}}}\left(\int_{\{u_{n}\geq 1\}}u_{n}^{\frac{\theta r_{j}}{p_{j}-r_{j}}}\right)^{\frac{p_{j}-r_{j}}{p_{j}}}\nonumber\\
&\leq C\left(\int_{\Omega}\mathcal{G}_{1}(u_{n})^{\frac{\theta r_{j}}{p_{j}-r_{j}}}\right)^{\frac{p_{j}-r_{j}}{p_{j}}}+C.
\end{align}
Based on the preceding inequalities, we infer
\begin{align*}
\prod_{j\in J}\left(\int_{\Omega}\vert D_{j}\mathcal{G}_{1}(u_{n})\vert^{r_{j}}\right)^{\frac{1}{r_{j}N}}&\leq C\prod_{j\in J}\left(\int_{\Omega}\mathcal{G}_{1}(u_{n})^{\frac{\theta r_{j}}{p_{j}-r_{j}}}\right)^{\frac{p_{j}-r_{j}}{r_{j}}\frac{1}{p_{j}N}}+C.
\end{align*}
Thus, By using Lemma \ref{L1} (see \eqref{(7)} with $v=\mathcal{G}_{1}(u_{n})$ and $\delta=\overline{r}^{*}$), we have
\begin{align}
\label{(30)}
\left(\int_{\Omega}\mathcal{G}_{1}(u_{n})^{\overline{r}^{*}}dx\right)^{\frac{1}{\overline{r}^{*}}}&\leq C\prod_{j\in J}\left(\int_{\Omega}\mathcal{G}_{1}(u_{n})^{\frac{\theta r_{j}}{p_{j}-r_{j}}}dx\right)^{\frac{p_{j}-r_{j}}{r_{j}}\frac{1}{p_{j}N}}+C.
\end{align}
Now, if we take $r_{j}=\lambda p_{j}$ with $\lambda\in[0,1)$, and we apply the assumption on $r_{j}$, we obtain
\begin{equation}
\label{(31)}
\frac{\theta r_{j}}{p_{j}-r_{j}}=\frac{\theta \lambda}{1-\lambda}=\frac{N\lambda \overline{p}}{N-\lambda\overline{p}}=\overline{r}^{*},
\end{equation}
and 
$$\lambda=\frac{N(\overline{p}-\theta)}{\overline{p}(N-\theta)}.$$
By \eqref{(30)} and \eqref{(31)}, we have
\begin{align}
\label{(32)}
\left(\int_{\Omega}\vert \mathcal{G}_{1}(u_{n})\vert^{\overline{r}^{*}}\right)^{\frac{N-\lambda\overline{p}}{N\lambda\overline{p}}}\leq C\left(\int_{\Omega} \mathcal{G}_{1}(u_{n})^{\overline{r}^{*}}\right)^{\frac{1-\lambda}{\lambda\overline{p}}}+C.
\end{align}
In addition,
\begin{equation}
\label{(33)}
N>\overline{p}\Leftrightarrow \frac{N-\lambda\overline{p}}{N\lambda\overline{p}}>\frac{1-\lambda}{\lambda\overline{p}}.
\end{equation}
As a result of \eqref{(32)} and \eqref{(33)}, $\{\mathcal{G}_{1}(u_{n})\}_{n}$ is uniformly bounded in $L^{\overline{r}^{*}}(\Omega)$. According to \eqref{(29)}, the sequence $\{\mathcal{G}_{1}(u_{n})\}_{n}$ is uniformly bounded in $W_{0}^{1,(r_{j})}(\Omega)$. 
\par Take $\varphi=\mathcal{T}_{k}(u_{n})$ ($k>0$) as a test function in the weak formulation \eqref{(22)}. Using \eqref{(5)} and \eqref{(6)},  we have
\begin{align*}
\alpha\sum_{j\in J}\int_{\{u_{n}\leq k\}}\vert D_{j}\mathcal{T}_{k}(u_{n})\vert^{p_{j}}+\sum_{j\in J}\int_{\{u_{n}\leq k\}}u_{n}^{q}\vert D_{j}\mathcal{T}_{k}(u_{n})\vert^{p_{j}}\leq k\Vert f\Vert_{L^{1}(\Omega)}.
\end{align*}
By leaving out the non-negative term on the left, we get
\begin{align}
\label{(34)}
\sum_{j\in J}\int_{\{u_{n}\leq k\}}\vert D_{j}\mathcal{T}_{k}(u_{n})\vert^{p_{j}}\leq kC,
\end{align}
taking $k=1$ in \eqref{(34)} we deduce that the sequence $\{\mathcal{T}_{1}(u_{n})\}_{n}$ is uniformly bounded in $W_{0}^{1,(p_{j})}(\Omega)$, hence in $W_{0}^{1,(\eta_{j})}(\Omega)$ (since $\eta_{j}\leq p_{j}$ for any $j\in J$), we obtain the sequence $\{u_{n}\}_{n}$ is uniformly bounded in $W_{0}^{1,(\eta_{j})}(\Omega)$ (Since $u_{n}=\mathcal{G}_{1}(u_{n})+\mathcal{T}_{1}(u_{n})$). 

\item[ii)] If $1-\theta<q\leq1$. Let us consider $\varphi=1-(1+u_{n})^{1-\lambda}$, as a test function in the weak formulation \eqref{(22)}, we obtain
\begin{align*}
&(\lambda-1)\sum_{j\in J}\int_{\Omega}\frac{(a(x)+u_{n}^{q})}{(1+u_{n})^{\lambda}}\vert D_{j}u_{n}\vert^{p_{j}}+\sum_{j\in J}\int_{\Omega}\frac{b(x)u_{n}\vert D_{j}u_{n}\vert^{p_{j}}}{\left(u_{n}+\frac{1}{n}\right)^{\theta+1}}\left[1-(1+u_{n})^{1-\lambda}\right]\\
&\quad=\int_{\Omega} f_{n}\left[1-(1+u_{n})^{1-\lambda}\right].
\end{align*}
By \eqref{(5)}, \eqref{(6)},  dropping positive terms and the fact that $1-(1+u_{n})^{1-\lambda}<1$, we have
\begin{align*}
\sum_{j\in J}\int_{\Omega}\frac{\alpha+u_{n}^{q}}{(1+u_{n})^{\lambda}}\vert D_{j}u_{n}\vert^{p_{j}} &\leq\frac{1}{\lambda-1} \Vert f\Vert_{L^{1}(\Omega)},
\end{align*}
since $q\leq1$, then $\min\{\alpha,1\}(1+u_{n})^{q}\leq \alpha+u_{n}^{q}$, this implies that
\begin{align}
\label{(35)}
\sum_{j\in J}\int_{\Omega}\frac{\vert D_{j}u_{n}\vert^{p_{j}}}{(1+u_{n})^{\lambda-q}}\leq C.
\end{align}
If $r_{j}<p_{j}$ for all $j\in J$, using H\"older's inequality and \eqref{(35)}, we get
\begin{align}
\label{(36)}
\int_{\Omega}\vert D_{j}u_{n}\vert^{r_{j}}&=\int_{\Omega}\frac{\vert D_{j}u_{n}\vert^{r_{j}}}{(1+u_{n})^{\frac{r_{j}}{p_{j}}(\lambda-q)}}(1+u_{n})^{\frac{r_{j}}{p_{j}}(\lambda-q)}\nonumber\\
&\leq C\left(\int_{\Omega}(1+u_{n})^{\frac{r_{j}(\lambda-q)}{p_{j}-r_{j}}}\right)^{\frac{p_{j}-r_{j}}{p_{j}}}
\end{align}
From the previous inequalities, we deduce
\begin{align*}
\prod_{j\in J}\left(\int_{\Omega}\vert D_{j}u_{n}\vert^{r_{j}}\right)^{\frac{1}{r_{j}N}}&\leq C\prod_{j\in J}\left(\int_{\Omega}u_{n}^{\frac{r_{j}(\lambda-q)}{p_{j}-r_{j}}}\right)^{\frac{p_{j}-r_{j}}{r_{j}}\frac{1}{p_{j}N}}+C.
\end{align*}
Thus, By using Lemma \ref{L1} (see \eqref{(7)} with $v=u_{n}$ and $\delta=\overline{r}^{*}$), we have
\begin{align}
\label{(37)}
\left(\int_{\Omega}u_{n}^{\overline{r}^{*}}\right)^{\frac{1}{\overline{r}^{*}}}&\leq C\prod_{j\in J}\left(\int_{\Omega}u_{n}^{\frac{r_{j}(\lambda-q)}{p_{j}-r_{j}}}\right)^{\frac{p_{j}-r_{j}}{\eta_{j}}\frac{1}{p_{j}N}}+C.
\end{align}
We put $r_{j}=\gamma p_{j}$ with $\gamma\in (0,1)$ and using the assumption made on $r_{j}$, we set
\begin{align}
\label{(38)}
\frac{r_{j}(\lambda-q)}{p_{j}-r_{j}}=\frac{\gamma(\lambda-q)}{1-\gamma}=\overline{r}^{*}=\frac{N\gamma \overline{p}}{N-\gamma\overline{p}},
\end{align}
we get
\begin{align*}
&\lambda=\frac{(\overline{p}-\overline{r})N\overline{p}+q\overline{p}(N-\overline{r})}{N-\overline{r}},
\end{align*}
and since $\lambda>1$, on has 
\begin{align}
\label{(39)}
r_{j}<\frac{N(\overline{p}-1+q)}{\overline{p}(N+q-1)}p_{j},\quad \forall\; j\in J.
\end{align}
By \eqref{(37)} and \eqref{(38)}, we have
\begin{align}
\label{(40)}
\left(\int_{\Omega}u_{n}^{\overline{r}^{*}}\right)^{\frac{N-\gamma\overline{p}}{N\gamma\overline{p}}}\leq C\left(\int_{\Omega} u_{n}^{\overline{r}^{*}}\right)^{\frac{1-\gamma}{\gamma\overline{p}}}+C.
\end{align}
In addition,
\begin{equation}
\label{(41)}
N>\overline{p}\Leftrightarrow \frac{N-\gamma\overline{p}}{N\gamma\overline{p}}>\frac{1-\lambda}{\gamma\overline{p}}.
\end{equation}
Hence, it follows from \eqref{(40)} and \eqref{(41)} that $\{u_{n}\}_{n}$ is bounded in $L^{\overline{r}^{*}}(\Omega)$. According to \eqref{(36)} along with \eqref{(37)} we deduce that the sequence $\{u_{n}\}_{n}$ is bounded in $W_{0}^{1,(r_{j})}(\Omega)$.

\item[iii)] If $q > 1$, we choose the test function $\varphi = 1 - (1 + u_{n})^{1-q}$, which yields
\begin{align*}
&(q-1)\sum_{j\in J}\int_{\Omega}\frac{(a(x)+u_{n}^{q})}{(1+u_{n})^{q}}\vert D_{j}u_{n}\vert^{p_{j}}\\&\qquad+\sum_{j\in J}\int_{\Omega}b(x)\frac{u_{n}\vert D_{j}u_{n}\vert^{p_{j}}}{\left(u_{n}+\frac{1}{n}\right)^{\theta+1}}\left[1-(1+u_{n})^{1-q}\right]\\
&\quad=\int_{\Omega} f_{n}\left[1-(1+u_{n})^{1-q}\right].
\end{align*}
By referencing conditions \eqref{(5)} and \eqref{(6)}, omitting the positive terms, and noting that $1 - (1 + u_{n})^{1 - q} < 1$, we can derive
\begin{align*}
\sum_{j\in J}\int_{\Omega}\frac{\alpha+u_{n}^{q}}{(1+u_{n})^{q}}\vert D_{j}u_{n}\vert^{p_{j}} &\leq\frac{1}{q-1} \Vert f\Vert_{L^{1}(\Omega)},
\end{align*}
since $q > 1$, we can establish that $(1 + u_{n})^{q} \leq 2^{q-1}\left[1 + (u_{n})^{q}\right]$. This inequality implies that
\begin{align*}
\frac{\min\{\alpha,1\}}{2^{q-1}}\sum_{j\in J}\int_{\Omega}\vert D_{j}u_{n}\vert^{p_{j}}\leq C,
\end{align*}
from which the boundedness of $u_{n}$ in $W_{0}^{1,(p_{j})}(\Omega)$ follows.
\end{itemize}
\end{proof}

The following lemma will be of central importance in the remainder of the proof, and it will be especially needed in the passage to the limit.
\begin{lem}\label{lmq}
Assume that the assumptions  \eqref{(2)}-\eqref{(6)} hold true.
Then, we have
\begin{equation}\label{ugest}
\int_{\Omega} u_{n}^{q\sigma_j}\left\vert D_{j} u_{n}\right\vert^{\sigma_j(p_{j}-1)}\;dx \leq C\quad\quad \forall \sigma_j<p_j^\prime.
\end{equation}
\end{lem}
\begin{proof}
 Taking $ \varphi(u_n)=1-\frac{1}{(1-u_n)^\lambda},$ with $\lambda>1,$ as test function in the weak formulation \eqref{(21)}, we get
\begin{equation*}
(\lambda-1)\sum_{j\in J}\int_{\Omega} \frac{[a(x)+u_n^q]\vert D_j u_n\vert^{p_{j}}}{(1+u_n)^\lambda}+\sum_{j\in J}\int_{\Omega}b(x)\frac{u_{n}\vert D_{j}u_{n}\vert^{p_{j}}}{\left(  u_{n}+\frac{1}{n}\right)^{\theta+1}}\varphi =\int_\Omega f_n \varphi(u_n).
\end{equation*}
We recall that there exists a positive constant $\mathcal{C}$ such that
$$(a(x)+u_n^q) \geq \mathcal{C}(1+u_n)^q,\quad \quad \forall q>1.$$ 
Dropping the positive term and using the fact that $\vert \varphi(u_n)\vert\leq 1,$ we obtain
\begin{equation}\label{est1}
\sum_{j\in J}\int_{\Omega} \frac{\vert D_j u_n\vert^{p_j}}{(1+u_n)^{\lambda-q}}\leq\frac{1}{\mathcal{C}(\lambda-1)}\int_\Omega \vert f_n\vert .
\end{equation}
For $1 < \sigma_{j} < p'_{j} = \frac{p_{j}}{p_{j} - 1}$, an application of H\"older's inequality together with estimate \eqref{est1} yields, for every $j \in J$,
\begin{align*}
\int_{\Omega} u_{n}^{q \sigma_{j}(p_{j}-1)}\left\vert D_{j} u_{n}\right\vert^{\sigma_{j}(p_{j}-1)} \leq C\left( \int_{\Omega} (1+u_{n})^{\frac{\sigma_{j}(p_{j}-1)(\lambda-q+p_{j} q)}{p_{j}-\sigma_{j}(p_{j}-1)}}\right)^{\frac{p_{j}-\sigma_{j}(p_{j}-1)}{p_{j}}}.
\end{align*}
We now introduce the parameter $\sigma_{j} = \theta\frac{p_{j}}{p_{j}-1}$, with $\theta \in (0,1)$. Under this choice, one readily verifies that
$$
\frac{p_{j}-\sigma_{j}(p_{j}-1)}{p_{j}} = 1 - \theta,\quad \forall\; j\in J,$$
and
$$ \frac{\sigma_{j}(p_{j}-1)(\lambda-q+p_{j} q)}{p_{j}-\sigma_{j}(p_{j}-1)} = \frac{\theta}{1-\theta}[\lambda+q(p_{j}-1)],\quad \forall\; j\in J.
$$
Substituting into the previous estimate leads to
\begin{equation}
\label{E11}
\int_{\Omega} u_{n}^{q \sigma_{j}(p_{j}-1)}\left\vert D_{j} u_{n}\right\vert^{\sigma_{j}(p_{j}-1)} \leq C\left( \int_{\Omega}(1+ u_{n})^{\frac{\theta}{1-\theta}[\lambda+q(p_{j}-1)]}\right)^{1-\theta}, 
\end{equation}
for every $j\in J$. Hence, we deduce the following inequality:
\begin{align*}
&\prod_{i\in J}\left(\int_{\Omega} u_{n}^{q \sigma_{j}(p_{j}-1)}\left\vert D_{j} u_{n}\right\vert^{\sigma_{j}(p_{j}-1)}\;dx\right)^{\frac{1}{\theta p_{j}}}\\
&\quad \leq C\prod_{i\in J}\left( \int_{\Omega} (1+u_{n})^{\frac{\theta}{1-\theta}[\lambda+q(p_{j}-1)]}\;dx\right)^{\frac{1-\theta}{\theta p_{j}}}.
\end{align*}
By invoking the anisotropic inequality \eqref{(9)} with $\delta_{j} = \sigma_{j}(p_{j}-1) \leq p_{j}$ (due to $\theta < 1$), $t_{j} = q \geq 0$, and $\overline{\delta} = \theta \overline{p}$, we obtain
\begin{equation}\label{E12}
\left(\int_{\Omega}u_{n}^{r}\right)^{\frac{N}{\theta\overline{p}}-1} \leq C\prod_{i\in J}\left( \int_{\Omega} (1+u_{n})^{\frac{\theta}{1-\theta}[\lambda+q(p_{j}-1)]}\right)^{\frac{1-\theta}{\theta p_{j}}}.
\end{equation}
In order for the preceding inequality to hold uniformly for all $j \in J$, we must enforce that
$$
r = \frac{\theta}{1-\theta}[\lambda+q(p_{j}-1)], \quad \forall\, j\in J.
$$
This leads to the following system of equations:
\begin{align}
& r=\frac{1+q}{b_{j}(N-1)-1+\frac{1}{\theta p_{j}}}, \quad \forall j\in J, \label{S1} \\
& r=\frac{\theta}{1-\theta}[\lambda+q(p_{j}-1)], \quad \forall j\in J, \label{S2} \\
&\sum_{i\in J}b_{j}=1, \quad b_{j} \geq 0, \quad \forall j\in J. \label{S3}
\end{align}
By combining equations \eqref{S1} and \eqref{S3}, one deduces
$$
r = \frac{N(1+q)\theta \overline{p}}{N-\theta\overline{p}}.
$$
In parallel, equation \eqref{S2} yields
$$
\theta = \frac{N(\overline{p}-\lambda+q)}{\overline{p}[N(1+q)-q\overline{p}-\lambda+q]}.
$$
Since $q \geq 0$ and $\lambda > 1$, it follows that $0 < \theta < 1$. Consequently, by combining these expressions, we arrive at
$$
r = \frac{N(\overline{p}-\lambda+q)}{N-\overline{p}}.
$$
Using this value of $r$ in inequality \eqref{E12}, we find
$$
\left(\int_{\Omega}u_{n}^{r}\right)^{\frac{N}{\theta\overline{p}}-1} \leq C\left(\int_{\Omega}u_{n}^{r}\right)^{\frac{(1-\theta)N}{\theta \overline{p}}}.
$$
Since the assumption $\overline{p} < N$ ensures that the left-hand exponent exceeds the right-hand one, we conclude that $\{u_{n}\}$ is bounded in $L^{r}(\Omega)$. This bound, together with estimate \eqref{E11}, implies that
$$
\int_{\Omega} u_{n}^{q \sigma_{j}(p_{j}-1)}\left\vert D_{j} u_{n}\right\vert^{\sigma_{j}(p_{j}-1)} \leq C, \quad \forall\, j\in J.
$$
Given that $p_{j} > 2$ and $\sigma_{j} > 1$ for all $j\in\mathcal{E}$, we apply Hölder’s inequality to obtain
\begin{align*}
\int_{\Omega} u_{n}^{q \sigma_{j}} \left\vert D_{j} u_{n}\right\vert^{\sigma_{j}(p_{j}-1)} 
&= \int_{\left\{u_{n}<1\right\}} u_{n}^{q \sigma_{j}} \left\vert D_{j} u_{n}\right\vert^{\sigma_{j}(p_{j}-1)} \\
&\qquad+ \int_{\left\{u_{n} \geq 1\right\}} u_{n}^{q \sigma_{j}} \left\vert D_{j} u_{n}\right\vert^{\sigma_{j}(p_{j}-1)}\\
&\leq \int_{\Omega} \left\vert D_{j} \mathcal{T}_{1}(u_{n})\right\vert^{\sigma_{j}(p_{j}-1)}
+ \int_{\Omega} u_{n}^{q \sigma_{j}(p_{j}-1)} \left\vert D_{j} u_{n}\right\vert^{\sigma_{j}(p_{j}-1)} \\
&\leq C \left(\int_{\Omega} \left\vert D_{j} \mathcal{T}_{1}(u_{n})\right\vert^{p_{j}} \right)^{\frac{\sigma_{j}(p_{j}-1)}{p_{j}}} + C \\
&\leq C.
\end{align*}
Thus, it follows that the sequence $\left\lbrace u_{n}^{q}\left\vert D_{j} u_{n}\right\vert^{p_{j}-1}\right\rbrace$ is uniformly bounded in $L^{\sigma_{j}}(\Omega)$ for all $1 < \sigma_{j} < p'_{j}$ and for each $j \in J$.
\end{proof}
Let us specify some useful notation we will use from now on. For any \(n\in \mathbb{N}^{*}\)  we define the following auxiliary functions:
\begin{equation*}
\mathcal{H}_{n}(l)=\int_{0}^{l}\frac{dt}{\left(t+\frac{1}{n}\right)^{\theta}},
\qquad
\mathcal{H}_{\infty}(l)=\int_{0}^{l}\frac{dt}{t^{\theta}},
\end{equation*}
and, let $\gamma=\frac{\nu}{\alpha}>0,$ we denote
\begin{equation*}
\Psi_{n}(l)=e^{-\gamma \mathcal{H}_{n}(l)},\qquad \Psi_{\infty}(l)=e^{-\gamma \mathcal{H}_{\infty}(l)}.
\end{equation*}
Note that, the function \(\mathcal{H}_{\infty}\) is well defined because \(\theta<1\). Moreover, it's clear that
$$\lim_{n\rightarrow \infty}\mathcal{H}_{n}(l)=\mathcal{H}_{\infty}(l)\text{ and }\lim_{n\rightarrow \infty}\Psi_{n}(l)=\Psi_{\infty}(l).$$
Observe that, for any \(\phi, u_{n}\in W_{0}^{1,(p_{j})}(\Omega)\cap L^{\infty}(\Omega)\) with \(\phi\ge0\) , we have \(\Psi_{n}(u_{n})\phi\in W_{0}^{1,(p_{j})}(\Omega)\cap L^{\infty}(\Omega)\).

Now, for fixed \(\lambda>0\) we define the Lipschitz cut-off function $\mathcal{F}_{\lambda}$ as follows
\begin{equation*}
\mathcal{F}_{\lambda}(l)=
\begin{cases}
1, & 0\le l<1,\\
-\dfrac{1}{\lambda}(l-1-\lambda), & 1\le l<1+\lambda,\\
0, & l\ge 1+\lambda.
\end{cases}
\end{equation*}
We stress that 
\begin{equation}\label{dr}
\mathcal{F}^{\prime}_{\lambda}(l):=\frac{-1}{\lambda}\chi_{\lbrace 1\le l<1+\lambda\rbrace}<0.
\end{equation}
We now introduce a lemma that will affirm the positivity of limit solution $u$.
\begin{lem}
\label{Lemma6.2}
Assume that the assumptions  \eqref{(2)}-\eqref{(6)} hold true. Then, \(u>0\) in \(\Omega\).
\end{lem}
\begin{proof}
Let \(\phi\in W_{0}^{1,(p_{j})}(\Omega)\cap L^{\infty}(\Omega)\) with \(\phi\ge0\). Take
\(\varphi=\Psi_{n}(u_{n})\phi\) as a test function in \eqref{(22)}. Thus, since
\begin{equation*}
D_{j}\varphi=-\gamma \frac{D_{j}u_{n}}{\left(u_{n}+\frac{1}{n}\right)^{\theta+1}}\Psi_{n}(u_{n})\phi+\Psi_{n}(u_{n})D_{j}\phi,\quad\quad \text{ for every }j\in J ,
\end{equation*}
we obtain
\begin{equation*}
\begin{aligned}
&\sum_{j\in J}\int_{\Omega}\left[a(x)+u_{n}^{q}\right] \vert D_{j}u_{n}\vert^{p_{j}-2}D_{j}u_{n} \Psi_{n}(u_{n})D_{j} \phi\\&
\qquad-\gamma\sum_{j\in J}\int_{\Omega} \left[a(x)+u_{n}^{q}\right] \vert D_{j}u_{n}\vert^{p_{j}}\frac{\Psi_{n}(u_{n})\phi}{\left(u_{n}+\frac{1}{n}\right)^\theta} \\
&\qquad+ \sum_{j\in J}\int_{\Omega}b(x)\vert D_{j} u_{n}\vert^{p_{j}}
\frac{u_{n}}{\left( u_{n}+\frac{1}{n}\right)^{\theta+1}}\Psi_{n}(u_{n})\phi\\&\quad=\int_{\Omega} T_n(f) \Psi_{n}(u_{n})\phi.
\end{aligned}
\end{equation*}
Since \(u_{n}\ge0\) and \(\mathcal{T}_{n}(f)\ge \mathcal{T}_{1}(f)\) for \(n\ge1\), recalling \eqref{(5)} and \eqref{(6)}, we get
\begin{equation}
\label{(6.9)'}
\sum_{j\in J}\int_{\Omega}\left[a(x)+u_{n}^{q}\right] \vert D_{j}u_{n}\vert^{p_{j}-2}D_{j}u_{n} \Psi_{n}(u_{n})D_{j} \phi
\ge \int_{\Omega} \mathcal{T}_{1}(f)\,\Psi_{n}(u_{n})\phi.
\end{equation}
Let \(\varphi\in W_{0}^{1,(p_{j})}(\Omega)\cap L^{\infty}(\Omega)\) with \(\varphi\ge0\), and take \(\phi=\mathcal{F}_{\lambda}(u_{n})\varphi\) in \eqref{(6.9)'}, we obtain
\begin{equation*}
\begin{aligned}
&\sum_{j\in J}\int_{\Omega}\left[a(x)+u_{n}^{q}\right] \vert D_{j}u_{n}\vert^{p_{j}-2}D_{j}u_{n} \Psi_{n}(u_{n})\mathcal{F}_{\lambda}(u_{n}) D_{j} \varphi\\&\qquad-\frac{1}{\lambda}\sum_{j\in J}\int_{\Omega}\left[a(x)+u_{n}^{q}\right] \vert D_{j}u_{n}\vert^{p_{j}}\Psi_{n}(u_{n}) \varphi\chi_{\lbrace 1\le u_{n}<1+\lambda\rbrace}\\&\quad
\ge \int_{\Omega} \mathcal{T}_{1}(f)\,\Psi_{n}(u_{n})\phi.
\end{aligned}
\end{equation*}
By \eqref{dr}, we deduce that
\begin{equation*}
\begin{aligned}
\label{(6.10)'}
&\sum_{j\in J}\int_{\Omega}\left[a(x)+u_{n}^{q}\right] \vert D_{j}u_{n}\vert^{p_{j}-2}D_{j}u_{n} \Psi_{n}(u_{n})\mathcal{F}_{\lambda}(u_{n}) D_{j} \varphi\ge \int_{\Omega} \mathcal{T}_{1}(f)\,\Psi_{n}(u_{n})\mathcal{F}_{\lambda}(u_{n})\varphi.
\end{aligned}
\end{equation*}
Letting \(\lambda\to0\), we derive that 
\begin{align}
\label{(6.12)'}
&\sum_{j\in J}\int_{\Omega}\left[a(x)+\mathcal{T}_{1}(u_{n})^{q}\right] \vert D_{j}\mathcal{T}_{1}(u_{n})\vert^{p_{j}-2}D_{j}\mathcal{T}_{1}(u_{n}) \Psi_{n}(\mathcal{T}_{1}(u_{n}))D_{j} \varphi\nonumber\\
&\quad\ge \int_{\lbrace 0\le u_{n}<1 \rbrace} \mathcal{T}_{1}(f)\,\Psi_{n}(\mathcal{T}_{1}(u_{n}))\varphi.
\end{align}
By \eqref{(34)} we have
\begin{align}
\label{(6.13)'}
&\left[a(x)+\mathcal{T}_{1}(u_{n})^{q}\right] \vert D_{j}\mathcal{T}_{1}(u_{n})\vert^{p_{j}-2}D_{j}\mathcal{T}_{1}(u_{n})\nonumber\\
&\quad \rightharpoonup \left[a(x)+\mathcal{T}_{1}(u)^{q}\right] \vert D_{j}\mathcal{T}_{1}(u)\vert^{p_{j}-2}D_{j}\mathcal{T}_{1}(u)
\quad\text{in }L^{(p_{j}^\prime)}(\Omega).
\end{align}
As a consequence, we can pass to the limit, as \(n\to\infty\), in \eqref{(6.12)'} to get
\begin{align}
\label{(6.14)'}
&\sum_{j\in J}\int_{\Omega}\left[a(x)+\mathcal{T}_{1}(u)^{q}\right] \vert D_{j}\mathcal{T}_{1}(u)\vert^{p_{j}-2}D_{j}\mathcal{T}_{1}(u) \Psi_{\infty}(\mathcal{T}_{1}(u))D_{j} \varphi\nonumber\\
&\quad\ge \int_{\lbrace 0\le u\leq 1 \rbrace} \mathcal{T}_{1}(f)\,\Psi_{\infty}(\mathcal{T}_{1}(u))\varphi.
\end{align}
for all nonnegative \(\varphi\in W_{0}^{1,(p_{j})}(\Omega)\).

For each \(j\in J\), we define
\begin{equation*}
\mathcal{V}(l)=\int_{0}^{\mathcal{T}_{1}(u(l))}e^{-\frac{\gamma}{p_{j}-1}\mathcal{H}_{\infty}(t)}\,dt.
\end{equation*}
Consequently, by the chain rule,
\begin{equation*}
\vert D_{j}\mathcal{V}(\mathcal{T}_{1}(u))\vert^{p_{j}-2}D_{j}\mathcal{V}(\mathcal{T}_{1}(u))
= \Psi_{\infty}(\mathcal{T}_{1}(u))\,\vert D_{j}\mathcal{T}_{1}(u)\vert^{p_{j}-2}D_{j}\mathcal{T}_{1}(u).
\end{equation*}
Thus \eqref{(6.14)'} becomes, by using that $\Psi_{\infty}(\mathcal{T}_{1}(u))>\Psi_{\infty}(1)$,
\begin{equation*}
\sum_{j\in J}\int_{\Omega}\left[a(x)+\mathcal{T}_{1}(u)^{q}\right]\vert D_{j}\mathcal{V}(\mathcal{T}_{1}(u))\vert^{p_{j}-2}D_{j}\mathcal{V}(\mathcal{T}_{1}(u)) D_{j}\varphi
\ge\int_{\Omega} h(x)\varphi,
\end{equation*}
where 
\begin{equation*}
h(x)=\mathcal{T}_{1}(f)\,\Psi_{\infty}(1)\chi_{\{0\le u\leq 1\}}>0.
\end{equation*}
Hence, since $0<a(x)+\mathcal{T}_{1}(u)^{q}<\beta+1,$ \(\mathcal{V}(\mathcal{T}_{1}(u))\in W_{0}^{1,(p_{j})}(\Omega)\) is a weak super-solution of
\begin{equation*}
\begin{cases}
-\displaystyle\sum_{j\in J}D_{j}\left(\left[a(x)+\mathcal{T}_{1}(u)^{q}\right]\vert D_{j}\mathcal{Z}\vert^{p_{j}-2}D_{j}\mathcal{Z}\right)=h(x), & \text{in }\Omega,\\
\mathcal{Z}\in W_{0}^{1,(p_{j})}(\Omega). &
\end{cases}
\end{equation*}

By the comparison principle in \(W_{0}^{1,(p_{j})}(\Omega)\) (see \cite{R3}), there exists a positive weak super-solution \(Z\) of the above equation and \(\mathcal{V}(\mathcal{T}_{1}(u))\ge \mathcal{Z}\).  
By Theorem \ref{T1} and Corollary \ref{C1}, we have \(\mathcal{Z}>0\) in \(\Omega\), hence \(\mathcal{V}(\mathcal{T}_{1}(u))>0\) a.e. in \(\Omega\).  
Since \(\mathcal{V}\) is strictly increasing, it follows that \(\mathcal{T}_{1}(u)>0\) a.e. in \(\Omega\). Therefore, \(u>0\) in \(\Omega\).
\end{proof}
As a consequence of Lemma \ref{L5}, there exists a subsequence (not relabelled) and a function
\[
u\in W_0^{1,(r_{j})}(\Omega), \quad\quad \text{ for every }j\in J ,
\]
with the components \(r_{j}\) as in the  Theorem \ref{T1}, such that the following convergences hold.
\begin{equation}\label{uae}
u_n\longrightarrow u\qquad\text{ a.e. in \ }\Omega.
\end{equation}
\begin{equation}\label{wtk}
\mathcal{T}_{k}(u_n)\rightharpoonup \mathcal{T}_{k}(u)\qquad\text{ in }W_0^{1,(p_{j})}(\Omega),
\end{equation}
for every \(k>0\).

To proceed, we establish that the gradients of the approximating solutions $u_n$ converge
almost everywhere in $\Omega$. This convergence is essential, due to the
nonlinearity of the operator, in order to pass to the limit in the approximate
problems and recover the anisotropic limit equation.  In our case, we closely follow the strategy developed in the isotropic proof of \cite{R0}, adapting each step to the present anisotropic operator of
Leray--Lions type.  This careful adaptation allows us to recover the almost-everywhere convergence of $D_{j} u_n$, a fundamental ingredient for the
identification of the limit problem.


\begin{lem}
Assume that  the hypotheses \eqref{(2)}-\eqref{(6)} hold true, and let $\{u_n\}$ be the approximating solutions for the problem \eqref{(21)}. If  $u$ is the limit function of $\{u_n\}$  that satisfies the convergences \eqref{uae} and \eqref{wtk}. Then, up to a subsequence not relabelled,
\begin{equation}\label{gae}
D_{j} u_n \longrightarrow D_{j} u 
\qquad\text{a.e.\ in }\Omega,\quad \text{ for every }j\in J .
\end{equation}
\end{lem}
\begin{proof}
Fix $h,k>0$ and choose $\varphi = T_h\left(u_n-\mathcal{T}_{k}(u)\right)$ as test function in \eqref{(22)}. Using  \eqref{(5)} and  \eqref{(23)}, we get
\begin{align*}
&\alpha\int_{\Omega}\sum_{j\in J}\big\vert D_{j} T_h(u_n-\mathcal{T}_{k}(u))\big\vert^{p_{j}}
 \\
&\quad \le C\,h\|f\|_{L^1(\Omega)}
-\sum_{j\in J}\int_{\Omega}\left[a(x)+u_n^q\right] \vert D_{j} \mathcal{T}_{k}(u)\vert^{p_{j}-2}D_{j} \mathcal{T}_{k}(u)\,D_{j} T_h(u_n-\mathcal{T}_{k}(u))
\end{align*}
Set $M=h+k$. Note that,  for every $j\in J $ we have 
\[
D_{j} T_h(u_n-\mathcal{T}_{k}(u))\ne0 \quad\Longrightarrow\quad u_n\le M,
\]
hence on the support of $\nabla T_h(u_n-\mathcal{T}_{k}(u))$ we may replace $u_n^q$ by $(T_M(u_n))^q$. Therefore, we obtain 
\begin{align}
\label{eqf}
&\alpha\int_{\Omega}\sum_{j\in J}\big\vert D_{j} T_h(u_n-\mathcal{T}_{k}(u))\big\vert^{p_{j}}\nonumber\\
&\quad\le -\sum_{j\in J}\int_{\Omega}\left[a(x)+T_M(u_n)^q\right] \vert D_{j} \mathcal{T}_{k}(u)\vert^{p_{j}-2}D_{j} \mathcal{T}_{k}(u)\,D_{j} T_h(u_n-\mathcal{T}_{k}(u))\nonumber\\
&\qquad+C\,h\|f\|_{L^1(\Omega)}.
\end{align}

By \eqref{wtk}, we have 
\[
D_{j} T_h(u_n-\mathcal{T}_{k}(u)) \rightharpoonup D_{j} T_h(u-\mathcal{T}_{k}(u))\quad \text{in } (L^{p_{j}}(\Omega))^N, \quad \quad \text{ for every }j\in J 
\]
in addition, thanks to \eqref{uae}, we deduce that
\begin{align*}
&\left[a(x)+T_M(u_n)^q\right] \vert D_{j} \mathcal{T}_{k}(u)\vert^{p_{j}-2}D_{j} \mathcal{T}_{k}(u)\\ 
&\quad\rightarrow \left[a(x)+T_M(u)^q\right] \vert D_{j} \mathcal{T}_{k}(u)\vert^{p_{j}-2}D_{j} \mathcal{T}_{k}(u)\quad \text{in } (L^{p_{j}^\prime}(\Omega))^N,\quad \quad \text{ for every }j\in J.
\end{align*}
Now, since $D_{j} \mathcal{T}_{k}(u)\cdot D_{j} T_h(u-\mathcal{T}_{k}(u))\equiv0$ for every $j\in J,$  it follows that
\begin{align*}
&\lim_{n\to\infty}\sum_{j\in J}\int_{\Omega}\left[a(x)+(T_M(u_n))^q\right] \vert D_{j} \mathcal{T}_{k}(u)\vert^{p_{j}-2}D_{j} \mathcal{T}_{k}(u)\,D_{j} T_h(u_n-\mathcal{T}_{k}(u))\\&\quad=\sum_{j\in J}\int_{\Omega}\left[a(x)+(T_M(u))^q\right] \vert D_{j} \mathcal{T}_{k}(u)\vert^{p_{j}-2}D_{j} \mathcal{T}_{k}(u)\,D_{j} T_h(u-\mathcal{T}_{k}(u))=0,
\end{align*}
which implies, after passing to the limit in \eqref{eqf}, that
\begin{equation}\label{eqff}
\alpha\limsup_{n\to\infty}\sum_{j\in J}\int_{\Omega}\big\vert D_{j} T_h(u_n-\mathcal{T}_{k}(u))\big\vert^{p_{j}}
\le C\,h\|f\|_{L^1(\Omega)}.
\end{equation}

 Let $\tau_{j}\in(1,r_{j})$ for all $j\in J$, where $r_{j}$ are as in the  Theorem \ref{T2}. We split
\begin{align*} 
 \sum_{j\in J}\int_{\Omega}\vert D_{j}(u_{n}-u)\vert^{\tau_{j}}&\leq\sum_{j\in J}\int_{\{\vert u_{n}-u\vert\leq h, \vert u\vert\leq k\}}\vert D_{j}(u_{n}-u)\vert^{\tau_{j}}\\
 & +\sum_{j\in J}\int_{\{\vert u_{n}-u\vert> h\}}\vert D_{j}(u_{n}-u)\vert^{\tau_{j}}
 \\
 &+\sum_{j\in J}\int_{\{\vert u_{n}-u\vert\leq h, \vert u\vert> k\}}\vert D_{j}(u_{n}-u)\vert^{\tau_{j}}.
\end{align*} 
Using the H\"older inequality with exponents $\frac{p_{j}}{\tau_{j}}$ and $\frac{r_{j}}{\tau_{j}}$
\begin{align*}
 \sum_{j\in J}\int_{\Omega}\vert D_{j}(u_{n}-u)\vert^{\tau_{j}}&\leq \left(\sum_{j\in J}\int_{\Omega}\big\vert D_{j} T_h(u_n-\mathcal{T}_{k}(u))\big\vert^{p_{j}}\right)^{\frac{\tau_{j}}{p_{j}}}\vert \Omega\vert^{1-\frac{\tau_{j}}{p_{j}}}\\
 &\quad +C\sum_{j\in J}\left[\mbox{meas}\{ \vert u_{n}-u\vert> h\}\right]^{1-\frac{\tau_{j}}{r_{j}}}\\
 &\quad +C \sum_{j\in J}\left[\mbox{meas}\{ \vert u_{n}-u\vert\leq h, \vert u\vert>k\}\right]^{1-\frac{\tau_{j}}{r_{j}}}.
\end{align*}
Thus, we deduce from the previous inequality and \eqref{uae} that
\begin{align*}
\limsup\limits_{h\rightarrow0}\limsup\limits_{n\rightarrow+\infty}\sum_{j\in J}\int_{\Omega}\vert D_{j}(u_{n}-u)\vert^{\tau_{j}}\leq C \sum_{j\in J}\left[\mbox{meas}\{  \vert u\vert>k\}\right]^{1-\frac{\tau_{j}}{r_{j}}}.
\end{align*}
Since 
$$\lim\limits_{k\rightarrow +\infty}\sum_{j\in J}\mbox{meas}\{  \vert u\vert>k\}=0,$$ 
we deduce that  
$$D_{j} u_{n}\rightarrow D_{j} u \text{ in $L^{\tau_{j}}(\Omega)$}, \quad \quad \text{ for every }j\in J $$
which gives, up to a subsequence, \eqref{gae}.
\end{proof}
Using \eqref{gae} and \eqref{uae}, together with the reverse Fatou's Lemma and Lebesgue’s Dominated Convergence Theorem, we can pass to the limit in \eqref{(23)} and obtain, for every \(k>0\),
\begin{equation}\label{kest}
\frac{1}{k}\sum_{j\in J}\int_\Omega \left[a(x)+u^{q}\right]\lvert D_{j} \mathcal{T}_{k}(u)\rvert^{p_{j}}
\le \int_\Omega f\frac{\mathcal{T}_{k}(u)}{k}.
\end{equation}
Fix $k>0$ and set
\begin{equation}\label{Qprt}
\mathcal{Q}(k)=\int_{\Omega} f\frac{\mathcal{T}_{k}(u)}{k}.
\end{equation}
By the dominated convergence theorem and the hypotheses on $f$, we obtain
\begin{equation}
\lim_{k\to\infty}\mathcal{Q}(k)=0.
\end{equation}
Therefore, we conclude
\begin{equation}\label{plim}
\lim_{k\to\infty}\frac{1}{k}\sum_{j\in J}\int_\Omega \left[a(x)+u^{q}\right]\lvert D_{j} \mathcal{T}_{k}(u)\rvert^{p_{j}}
=0.
\end{equation}
Moreover, by applying Fatou’s Lemma once again in view of \eqref{gae}, we may pass to the limit in estimates \eqref{(22)} and \eqref{ugest}, which yields
\begin{equation}\label{gl1}
\sum_{j\in J} b(x)\frac{\lvert D_{j}u\rvert^{p_{j}}}{u^{\theta}}  \in L^{1}(\Omega),
\end{equation}
\begin{equation}\label{pest}
\displaystyle \sum_{j\in J}\left[a(x)+u^{q}\right]\lvert D_{j}u\rvert^{p_{j}-2} D_{j}u \in (L^{\sigma_{j}}(\Omega))^{N} \quad \text{for all } \sigma_{j} < p_{j}'.
\end{equation}

To finish the proof of our main  Theorem, it remains to prove that $u$ is a distributional solution of the Problem \ref{(1)}. This is the goal of the next part.
\subsection{Passing to the limit}
Let $\mathcal{B}\in C^1(\mathbb{R})$ satisfy $0\le \mathcal{B}(l)\le1$ for all $l\in \mathbb{R}$ and
\begin{equation}
\begin{aligned}
\mathcal{B}(l)=
\begin{cases}
1 & \text{ if } 0\leq l\leq \frac{1}{2},\\
0 & \text{ if } l\geq 1
\end{cases}
\end{aligned}
\end{equation}

The proof will be completed in two steps.

\textbf{Step 1. The first inequality:} 

We fix $\psi \in W_0^{1,d_j}(\Omega),\; d_j>p_j$ for every $j\in J$, with $\psi \geq 0$ and take
\[
\varphi=\psi\, \mathcal{B}\left(\frac{u_{n}}{k}\right)
\]
as a test function in \eqref{(22)}. Since, for every $j\in J$,
\[
D_j\varphi= D_j\psi\, \mathcal{B}\left(\frac{u_{n}}{k}\right) + \frac{\psi}{k} \mathcal{B}^\prime\left(\frac{u_{n}}{k}\right) D_j \mathcal{T}_{k}\left(u_{n}\right).
\]
By the hypotheses on $\mathcal{B}$ this yields
\[
\begin{aligned}
&\sum_{j\in J}\int_{\Omega} \left[a(x)+\mathcal{T}_{k}(u_{n})^q\right]\,\vert D_j \mathcal{T}_{k}(u_{n})\vert^{p_j-2}D_j \mathcal{T}_{k}(u_{n})\,D_j\psi \; \mathcal{B}\left(\frac{u_{n}}{k}\right)\\
&\qquad +\frac{1}{k}\sum_{j\in J}\int_{\Omega}\left[a(x)+\mathcal{T}_{k}(u_{n})^q\right]\,\vert D_j \mathcal{T}_{k}(u_{n})\vert^{p_j}\,\psi\; \mathcal{B}^\prime\left(\frac{u_{n}}{k}\right)\\
&\qquad +\sum_{j\in J}\int_{\Omega}\frac{b(x)\,\mathcal{T}_{k}(u_{n})\,\vert D_j \mathcal{T}_{k}(u_{n})\vert^{p_j}}{\left(\mathcal{T}_{k}(u_{n})+\frac{1}{n}\right)^{\theta+1}}\;\psi\; \mathcal{B}\left(\frac{u_{n}}{k}\right)= \int_{\Omega}f_{n}\;\psi\; \mathcal{B}\left(\frac{u_{n}}{k}\right).
\end{aligned}
\]
Hence, using \eqref{(24)}, we have
\[
\begin{aligned}
&\sum_{j\in J}\int_{\Omega} \left[a(x)+\mathcal{T}_{k}(u_{n})^q\right]\vert D_j \mathcal{T}_{k}(u_{n})\vert^{p_j-2}D_j \mathcal{T}_{k}(u_{n})D_j\psi \; \mathcal{B}\left(\frac{u_{n}}{k}\right)\\
&\qquad +\sum_{j\in J}\int_{\Omega}\frac{b(x)\mathcal{T}_{k}(u_{n})\vert D_j \mathcal{T}_{k}(u_{n})\vert^{p_j}}{\left(u_{n}+\varepsilon\right)^{\theta+1}}\psi \mathcal{B}\left(\frac{u_{n}}{k}\right)\\
&\quad \le \int_{\Omega}f_{n}\psi \mathcal{B}\left(\frac{u_{n}}{k}\right)
+\|B'\|_{L^{\infty}(\mathbb{R})}\|\psi\|_{L^{\infty}(\Omega)}\int_{\Omega} f\frac{\mathcal{T}_{k}(u_{n})}{k}.
\end{aligned}
\]
Passing to the limit as $n\to \infty$ is justified by \eqref{uae}, \eqref{wtk}, \eqref{gae}, Fatou's lemma and the dominated convergence theorem. We thus obtain
\[
\begin{aligned}
&\sum_{j\in J}\int_{\Omega} \left[a(x)+\mathcal{T}_{k}(u)^q\right]\vert D_j \mathcal{T}_{k}(u)\vert^{p_j-2}D_j \mathcal{T}_{k}(u)D_j\psi  \mathcal{B}\left(\frac{u}{k}\right)\\
&\qquad +\sum_{j\in J}\int_{\Omega}\frac{b(x)\vert D_j \mathcal{T}_{k}(u)\vert^{p_j}}{\mathcal{T}_{k}(u)^\theta}\psi \mathcal{B}\left(\frac{u}{k}\right)\\
&\quad \le \int_{\Omega} f\psi \mathcal{B}\left(\frac{u}{k}\right)
+\|\mathcal{B}^{\prime}\|_{L^{\infty}(\mathbb{R})}\|\psi\|_{L^{\infty}(\Omega)}\mathcal{Q}(k),
\end{aligned}
\]
for all $\psi \in W_0^{1,d_j}(\Omega),\; d_j>p_j$ with $\psi\ge0$.

Letting $k\to\infty$ and using  \eqref{gl1}, \eqref{pest} and that $\mathcal{B}$ is bounded, we deduce
\[
\sum_{j\in J}\int_{\Omega}\left[a(x)+u^q\right]\vert D_j u\vert^{p_j-2}D_j u D_j\psi
+\sum_{j\in J}\int_{\Omega}\frac{b(x)\vert D_j u\vert^{p_j}}{u^\theta}\psi
\le \int_{\Omega} f\psi,
\]
for every $\psi \in W_0^{1,d_j}(\Omega),\; d_j>p_j$ with $\psi\ge0$.

\textbf{Step 2:  The second inequality:}

Let $\psi\in W_0^{1,p_j}(\Omega)\cap L^\infty(\Omega)$ satisfy $\psi\le0$.  Set
\[
\varphi=\psi\,\mathrm{e}^{-\gamma \mathcal{H}_{n}(u_{n})}\,\mathcal{B}\left(\frac{u_{n}}{k}\right)
\]
as test function in \eqref{(22)}. Since $\mathcal{B}\big(\tfrac{u_{n}}{k}\big)$ and $\mathcal{B}'\big(\tfrac{u_{n}}{k}\big)$ vanish on $\{u_{n}>k\}$, one obtains
\[
\begin{aligned}
&\sum_{j\in J}\int_{\Omega}\left[  a(x)+\mathcal{T}_{k}(u_{n})^q \right]\vert D_j \mathcal{T}_{k}(u_{n})\vert^{p_j-2}D_j \mathcal{T}_{k}(u_{n})D_j\psi\mathrm{e}^{-\gamma \mathcal{H}_{n}(u_{n})} \mathcal{B}\left(\frac{u_{n}}{k}\right)\\
&\qquad -\gamma\sum_{j\in J}\int_{\Omega}\left[a(x)+\mathcal{T}_{k}(u_{n})^q \right]\frac{\vert D_j \mathcal{T}_{k}(u_{n})\vert^{p_j}}{\left(\mathcal{T}_{k}(u_{n})+\frac{1}{n}\right)^\theta}\psi\mathrm{e}^{-\gamma \mathcal{H}_{n}(u_{n})} \mathcal{B}\left(\frac{u_{n}}{k}\right)\\
&\qquad +\frac{1}{k}\sum_{j\in J}\int_{\Omega}\left[a(x)+\mathcal{T}_{k}(u_{n})^q\right]\vert D_j \mathcal{T}_{k}(u_{n})\vert^{p_j}\psi\mathrm{e}^{-\gamma \mathcal{H}_{n}(u_{n})} \mathcal{B}^\prime\left(\frac{u_{n}}{k}\right)\\
&\qquad +\sum_{j\in J}\int_{\Omega}\frac{b(x)\mathcal{T}_{k}(u_{n})\vert D_j \mathcal{T}_{k}(u_{n})\vert^{p_j}}{\left(\mathcal{T}_{k}(u_{n})+\frac{1}{n}\right)^{\theta+1}}\psi\mathrm{e}^{-\gamma \mathcal{H}_{n}(u_{n})} \mathcal{B}\left(\frac{u_{n}}{k}\right)\\
&\quad=\int_{\Omega}f_{n}\psi\mathrm{e}^{-\gamma \mathcal{H}_{n}(u_{n})} \mathcal{B}\left(\frac{u_{n}}{k}\right).
\end{aligned}
\]
Note that, by the assumptions on $a$ and $b$, and since $\psi\le0$, we have for each $j\in J$
\[
\frac{\vert D_j \mathcal{T}_{k}(u_{n})\vert^{p_j}}{\left(\mathcal{T}_{k}(u_{n})+\frac{1}{n}\right)^\theta}\;\psi\;\mathrm{e}^{-\gamma \mathcal{H}_{n}(u_{n})} \mathcal{B}\left(\frac{u_{n}}{k}\right)
\left[\frac{b(x)\,\mathcal{T}_{k}(u_{n})}{\mathcal{T}_{k}(u_{n})+\frac{1}{n}}-\gamma\left[a(x)+u_{n}^q\right]\right]\ge0.
\]

Using \eqref{(24)}, \eqref{uae}, \eqref{wtk}, \eqref{gae}, \eqref{kest} and Fatou's lemma, we let $n\to \infty$ and obtain
\[
\begin{aligned}
&\sum_{j\in J}\int_{\Omega}\left[a(x)+\mathcal{T}_{k}(u)^q\right]\vert D_j \mathcal{T}_{k}(u)\vert^{p_j-2}D_j \mathcal{T}_{k}(u)D_j\psi\mathrm{e}^{-\gamma \mathcal{H}_{\infty}(u)} \mathcal{B}\left(\frac{u}{k}\right)\\
&\qquad-\gamma\sum_{j\in J}\int_{\Omega}\left[a(x)+\mathcal{T}_{k}(u)^q\right]\frac{\vert D_j \mathcal{T}_{k}(u)\vert^{p_j}}{\mathcal{T}_{k}(u)^\theta}\psi\mathrm{e}^{-\gamma \mathcal{H}_{\infty}(u)} \mathcal{B}\left(\frac{u}{k}\right)\\
&\qquad+\sum_{j\in J}\int_{\Omega}\frac{b(x)\vert D_j \mathcal{T}_{k}(u)\vert^{p_j}}{\mathcal{T}_{k}(u)^\theta}\psi\mathrm{e}^{-\gamma \mathcal{H}_{\infty}(u)} \mathcal{B}\left(\frac{u}{k}\right)\\
& \quad\le \int_{\Omega} f\psi\mathrm{e}^{-\gamma \mathcal{H}_{\infty}(u)} \mathcal{B}\left(\frac{u}{k}\right)+\|\mathcal{B}^{\prime}\|_{L^\infty(\mathbb{R})}\|\psi\|_{L^\infty(\Omega)}\mathcal{Q}(k).
\end{aligned}
\]
We now choose a particular $\psi$ and let $k\to\infty$. Take $k>0$ large enough such that
\[
\sigma(k)=\left(-\frac{\gamma(1-\theta)}{2}\log(\mathcal{Q}(k))\right)^{\frac{1}{1-\theta}}.
\]
Remark that, thanks to \eqref{Qprt}, we have 
 $$\lim_{k\to+\infty}\sigma(k)=+\infty.$$ 
  Note also that, by definition,
\begin{equation}\label{expQ}
\mathrm{e}^{\gamma \mathcal{H}_{\infty}(\sigma(k))}=\frac{1}{\sqrt{\mathcal{Q}(k)}}.
\end{equation}
Let $\Phi\in C_c^1(\Omega)$ with $\varphi\le0$. Since $u$ is strictly positive on compact subsets of $\Omega$, we have $u^{-\theta}\Phi\in L^\infty(\Omega)$, and therefore the negative function
\[
\psi=\mathrm{e}^{-\gamma \mathcal{H}_{\infty}(u)}
\mathcal{B}\left(\frac{u}{\sigma(k)}\right)\Phi
\]
belongs to $W_0^{1,p_j}(\Omega)\cap L^\infty(\Omega)$ (recall that $\mathcal{B}$ and $\mathcal{B}'$ have compact support). Using this $\psi$ in the previous inequality, cancelling common terms and applying \eqref{kest} and \eqref{expQ}, we obtain
\[
\begin{aligned}
&\sum_{j\in J}\int_{\Omega}\left[a(x)+\mathcal{T}_{k}(u)^q\right]\vert D_j \mathcal{T}_{k}(u)\vert^{p_j-2}D_j \mathcal{T}_{k}(u)\,D_j\Phi \mathcal{B}\left(\frac{u}{k}\right) \mathcal{B}\left(\frac{u}{\sigma(k)}\right)\\
&\qquad +\sum_{j\in J}\int_{\Omega}\frac{b(x)\vert D_j \mathcal{T}_{k}(u)\vert^{p_j}}{\mathcal{T}_{k}(u)^\theta}\Phi \mathcal{B}\left(\frac{u}{k}\right) \mathcal{B}\left(\frac{u}{\sigma(k)}\right)\\
& \quad\le \int_{\Omega} f\Phi \mathcal{B}\left(\frac{u}{k}\right) \mathcal{B}\left(\frac{u}{\sigma(k)}\right) +\|\mathcal{B}^{\prime}\|_{L^\infty(\mathbb{R})}\|\varphi\|_{L^\infty(\Omega)}\sqrt{\mathcal{Q}(k)}\\
&\qquad+\frac{1}{\sigma(k)}\|\mathcal{B}^{\prime}\|_{L^\infty(\mathbb{R})}\|\varphi\|_{L^\infty(\Omega)}\sum_{j\in J}\int_{\Omega}\left[a(x)+T_{\sigma(k)}(u)^q\right]\vert D_j \mathcal{T}_{\sigma(k)}(u)\vert^{p_j}\\
&\quad \le \int_{\Omega} f\Phi \mathcal{B}\left(\frac{u}{k}\right) \mathcal{B}\left(\frac{u}{\sigma(k)}\right)
+\|\mathcal{B}^{\prime}\|_{L^\infty(\mathbb{R})}\|\varphi\|_{L^\infty(\Omega)}\sqrt{\mathcal{Q}(k)}\\
&\qquad+\|\mathcal{B}^{\prime}\|_{L^\infty(\mathbb{R})}\|\varphi\|_{L^\infty(\Omega)}\mathcal{Q}(\sigma(k)).
\end{aligned}
\]
Finally, passing to the limit as $k\to\infty$ and using that $\mathcal{B}$ is bounded, together with \eqref{kest} and \eqref{plim}, we infer
\[
\sum_{j\in J}\int_{\Omega}\left[a(x)+u^q\right]\,|D_j u|^{p_j-2}D_j u\,D_j\Phi
+\sum_{j\in J}\int_{\Omega}\frac{b(x)\,\vert D_j u\vert^{p_j}}{u^\theta}\Phi
\le \int_{\Omega} f\Phi,
\]
for every $\Phi\in C_c^1(\Omega)$ with $\varphi\le0$. By density this inequality extends to all $\Phi\in  W_0^{1,d_j}(\Omega),\; d_j>p_j$ with $\varphi\le0$.
\[
\sum_{j\in J}\int_{\Omega}\left[a(x)+u^q\right]\vert D_j u\vert^{p_j-2}D_j u D_j\Phi
+\sum_{j\in J}\int_{\Omega}\frac{b(x)\vert D_j u\vert^{p_j}}{u^\theta}\Phi
\le \int_{\Omega} f\Phi.
\]
Putting together the results of both steps we conclude that
\[
\sum_{j\in J}\int_{\Omega}\left[a(x)+u^q\right]|D_j u|^{p_j-2}D_j u D_j\Phi
+\sum_{j\in J}\int_{\Omega}\frac{b(x)\vert D_j u\vert^{p_j}}{u^\theta}\Phi
\le \int_{\Omega} f\Phi,
\]
for every $\Phi\in W_0^{1,d_j}(\Omega),\; d_j>p_j.$ 
 
So that the proof of the theorem is complete.

\textbf{Acknowledgments}:
The authors are thankful to the anonymous referee for his/her careful reading
of the original manuscript, which led to substantial improvements.
\bigskip

\textbf{Data Availability Statement:} Our manuscript has no associated data.

\bigskip
\textbf{Conflict of interest statement:} The authors declare that there is no conflict of interest
regarding the publication of this paper.

\end{document}